\documentclass[a4paper,12pt]{article}
\usepackage[T1]{fontenc}
\usepackage{mathrsfs}
\usepackage{graphicx}
\usepackage{enumerate}
\usepackage{multicol}
\usepackage{color}
\usepackage{amsmath,amssymb}
\usepackage{amsthm}
\usepackage{textcomp}
\usepackage{mathabx}

\usepackage{hyperref}

\usepackage{times}
\usepackage[varg]{txfonts}

\usepackage[top=1in, bottom=1in, left=0.75in, right=0.75in]{geometry}

\theoremstyle{plain}
\newtheorem{thm}{Theorem}
\newtheorem{lem}{Lemma}
\newtheorem{cor}{Corollary}
\newtheorem{prop}{Proposition}

\theoremstyle{definition}
\newtheorem{dfn}{Definition}

\theoremstyle{remark}

\newcommand{\R}{\mathbb{R}}

\newcommand{\diff}[2]{\dfrac{\mathrm{d}^{#1}}{\mathrm{d}{#2}^{#1}}}
\newcommand{\supp}{\mathrm{supp}\,}

\title{A direct and algebraic characterization of  higher-order differential operators}
\author{Włodzimierz Fechner and Eszter Gselmann}

\begin{document}

\maketitle

\begin{abstract}

This paper presents an algebraic approach to characterizing higher-order differential operators. While the foundational Leibniz rule addresses first-order derivatives, its extension to higher orders typically involves identities relating multiple distinct operators. In contrast, we introduce a novel operator equation involving only a single $n$\textsuperscript{th}-order differential operator. We demonstrate that, under certain mild conditions, this equation serves to characterize such operators. Specifically, our results show that these higher-order differential operators can be identified as particular solutions to this single-operator identity. This approach provides a framework for understanding the algebraic structure of higher-order differential operators acting on function spaces.
\end{abstract}

\section{Introduction}

Let $\Omega \subset \mathbb{R}$ be a nonempty and open set and $k$ be a nonnegative integer. Let us consider the function space 
\[
 \mathscr{C}^{k}(\Omega)= 
 \left\{  f\colon \Omega \to \mathbb{R}\, \vert \, f \text{ is } k \text{ times continuously differentiable}\right\}. 
\]
For $k=0$, instead of $\mathscr{C}^{0}(\Omega)$ we simply write $\mathscr{C}(\Omega)$ for the linear spaces of all continuous functions $f\colon \Omega\to \mathbb{R}$. 

A fundamental property of the first-order derivative is the Leibniz rule, i.e. 
\[
\diff{}{x}(f\cdot g) = f \cdot \diff{}{x} g+ \diff{}{x}f \cdot g 
\qquad 
\left(f, g\in \mathscr{C}^{1}(\Omega)\right). 
\]

H.~König and V.~Milman \cite{KonMil11} investigated the extent to which the Leibniz rule characterizes the derivative. Their main result is presented in the theorem below (see also \cite[Theorem 3.1]{KonMil18}, \cite{GolSem96}).

\begin{thm}[König-Milman]
 Let $\Omega\subset \mathbb{R}$ be a nonempty and open set and $k$ be a nonnegative integer. Suppose that the operator $T\colon \mathscr{C}^k(\Omega)\to \mathscr{C}(\Omega)$ satisfies the Leibniz rule, i.e., 
 \begin{equation}\label{LR}
  T(f\cdot g)= f\cdot T(g)+T(f)\cdot g 
  \qquad 
  \left(f, g\in \mathscr{C}^k(\Omega)\right). 
 \end{equation}
Then there exist functions $c, d\in \mathscr{C}^k(\Omega)$ such that for all $f\in \mathscr{C}(\Omega)$ and $x\in \Omega$ 
\[
 T(f)(x)= c(x)\cdot f(x) \cdot \ln \left(\left|f(x)\right|\right)+ d(x) \cdot f'(x). 
\]
For $k=0$ we necessarily have $d=0$. 
Conversely, any such map $T$ satisfies \eqref{LR}. 
\end{thm}

The standard generalization of the Leibniz rule to the $n^{th}$ derivative is given by the identity:
 \begin{equation}\label{Leibnizn}
  \diff{n}{x}(f\cdot g)= \sum_{k=0}^{n}\binom{n}{k}\diff{k}{x}f\cdot \diff{n-k}{x}g
  \qquad
  \left(f, g \in \mathscr{C}^{n}(\Omega)\right).
 \end{equation}
A notable feature of this identity is that it involves not only the $n^{th}$-order differential operator $\diff{n}{x}$, but also all lower-order operators $\diff{k}{x}$ for $k=0, 1, \ldots, n-1$.
Consequently, the corresponding abstract operator equation,
 \begin{equation}\label{operatorn}
  T_{n}(f\cdot g)= \sum_{k=0}^{n}\binom{n}{k}T_{k}(f)\cdot T_{n-k}(g)
  \qquad
  \left(f, g \in \mathscr{C}^{n}(\Omega)\right),
 \end{equation}
is also not an equation characterizing a single operator $T_n$. Instead, it establishes a relationship involving the entire sequence of operators $T_{k}$ for $k=0, 1, \ldots, n$.
This raises the question: Is it possible to characterize the $n^{th}$-order differential operator $\diff{n}{x}$ using an operator equation that involves only this single operator?
To explore this, consider a general $n^{th}$-order linear differential operator $\mathbf{D}_n$ (for $n\in \mathbb{N}$) with continuous coefficients $c_{i}\in \mathscr{C}(\Omega)\, (i=1, \ldots, n)$, defined as:
 \[
  \mathbf{D}_{n}= \sum_{i=1}^{n}c_{i}\diff{i}{x}.
 \]
It can be shown by induction that such operators satisfy the following identity for any $f_{1}, \ldots, f_{n+1}\in \mathscr{C}^{n}(\Omega)$:
 \[
  \sum_{i=0}^{n}(-1)^{i}\sum_{\mathrm{card}(I)=i}\left(\prod_{j\in I}f_{j}\right)\mathbf{D}_{n}\left(\prod_{k\in \left\{1, \ldots, n+1 \right\}\setminus I}f_{k}\right)=0.
 \]
For $n=1$, this identity corresponds to the Leibniz rule for $\mathbf{D}_1$:
 \[
  \mathbf{D}_{1}(f_{1}\cdot f_{2})-f_{1}\cdot \mathbf{D}_{1}(f_{2})-f_{2}\cdot \mathbf{D}_{1}(f_{1})=0
  \qquad
  \left(f_{1}, f_{2}\in \mathscr{C}^{1}(\Omega)\right).
 \]
For $n=2$, the identity becomes:
 \begin{multline*}
 \mathbf{D}_{2}(f_{1}\cdot f_{2} \cdot f_{3}) - f_{1} \mathbf{D}_{2}(f_{2}\cdot f_{3}) - f_{2} \mathbf{D}_{2}(f_{1}\cdot f_{3}) - f_{3} \mathbf{D}_{2}(f_{1} \cdot f_{2})
 \\
 + f_{1}\cdot f_{2}  \mathbf{D}_{2}(f_{3}) + f_{1}\cdot f_{3}  \mathbf{D}_{2}(f_{2}) +f_{2}\cdot f_{3}  \mathbf{D}_{2}(f_{1}) =0
 \\
  \left(f_{1}, f_{2}, f_{3}\in \mathscr{C}^{2}(\Omega)\right).
 \end{multline*}

In paper \cite{FecGse25}, we examined the operator equation corresponding to this identity and proved that, under certain mild conditions, the solutions to this equation are precisely the differential operators of order at most two. 
Building upon these findings, the current paper aims to extend our characterization to higher-order differential operators using direct algebraic methods.  

Let $n$ be arbitrarily fixed positive integer, while $k$ be a fixed nonnegative integer and $\Omega\subset \mathbb{R}$ be a nonempty and open set. In what follows, we will study operators $D\colon \mathscr{C}^{k}(\Omega)\to \mathscr{C}(\Omega)$ that fulfill 
\begin{equation}\label{id_n}
 \sum_{i=0}^{n}(-1)^{i}\sum_{\mathrm{card}(I)=i}\left(\prod_{j\in I}f_{j}\right)D\left(\prod_{k\in \left\{1, \ldots, n+1 \right\}\setminus I}f_{k}\right)=0
\end{equation}
for all $f_{1}, \ldots, f_{n+1}\in \mathscr{C}^{k}(\Omega)$. It is important to note that, unless explicitly stated, we do not assume the operator $D$ to be linear. As we will show, this identity \eqref{id_n} will prove to be a suitable tool for characterizing higher-order linear differential operators in function spaces

In addition to the study of equation \eqref{id_n}, we will also investigate the related operator identity
\begin{equation}\label{id_single}
 \sum_{i=0}^{n}(-1)^{i}\binom{n+1}{i}f^{i}\cdot D\left(f^{n+1-i}\right)=0 
 \qquad 
 \left(f\in \mathscr{C}^{n}(\Omega)\right)
\end{equation}
for the operator $D\colon \mathscr{C}^{k}(\Omega)\to \mathscr{C}(\Omega)$, where $f^{i}$ denotes the $i$\textsuperscript{th} power of the function $f$. 

\section{Results}

In our first lemma we will prove a localization property for operators satisfying equation \eqref{id_n}.

\begin{lem}\label{lem_loc_interval}
 Let $n$ be a fixed positive integer, $k$ be a fixed nonnegative integer, and $\Omega\subset \mathbb{R}$ be a nonempty open set.
 Suppose that the operator  $D\colon \mathscr{C}^{k}(\Omega)\to \mathscr{C}(\Omega)$ satisfies equation \eqref{id_n}. Then $D(\mathbf{1})= D(\mathbf{-1})=0$. Furthermore, $D$ is \emph{localized on intervals}, meaning that if $J\subset \Omega$ is an open interval and $f_{1}, f_{2}\in \mathscr{C}^{k}(\Omega)$ are functions such that $f_{1}\vert_{J}= f_{2}\vert_{J}$, then $D(f_{1})\vert_{J}= D(f_{2})\vert_{J}$.
\end{lem}

\begin{proof}
 First, we will show that $D(\mathbf{1})=0$. Substituting $f_{i}= \mathbf{1}$ for all $i=1, \ldots, n+1$ into equation \eqref{id_n}, we obtain
 \[
  \sum_{i=0}^{n}(-1)^{i}\sum_{\mathrm{card}(I)=i}\left(\prod_{j\in I}\mathbf{1}\right)D\left(\prod_{k\in \left\{1, \ldots, n+1 \right\}\setminus I}\mathbf{1}\right)=0.
 \]
 Since $\mathbf{1}^m = \mathbf{1}$ for any nonnegative integer $m$, this simplifies to
 \[
  \sum_{i=0}^{n}(-1)^{i}\binom{n+1}{i}D(\mathbf{1})=0,
 \]
 as there are $\binom{n+1}{i}$ subsets $I$ of $\{1, \ldots, n+1\}$ with cardinality $i$. Factoring out $D(\mathbf{1})$, we have
 \[
  \left(\sum_{i=0}^{n}(-1)^{i}\binom{n+1}{i}\right)D(\mathbf{1})=0.
 \]
 Using the binomial identity $$\sum_{i=0}^{n+1}(-1)^{i}\binom{n+1}{i} = (1-1)^{n+1} = 0,$$ we deduce that $$\sum_{i=0}^{n}(-1)^{i}\binom{n+1}{i} = -(-1)^{n+1}\binom{n+1}{n+1} = (-1)^{n}.$$ Therefore, $(-1)^{n}D(\mathbf{1})=0$, which implies $D(\mathbf{1})=0$.

 Next, apply \eqref{id_n} with $f_{i}=-\mathbf{1}$ for $i=1, \ldots, n+1$ to arrive at 
\[
 \sum_{i=0}^{n}(-1)^{i}\binom{n+1}{i}(-\mathbf{1})^{i}D((-\mathbf{1})^{n+1-i})=0, 
\]
i.e., 
\[
 \sum_{i=0}^{n}D((-\mathbf{1})^{n+1-i})=0, 
\]
from which $D(-\mathbf{1})=0$ follows, as $D(\mathbf{1})=0$.

 Now, we will prove that $D$ is localized on intervals. Let $J\subset \Omega$ be an open interval, and assume that $\varphi_{1}, \varphi_{2}\in \mathscr{C}^{k}(\Omega)$ are such that $\varphi_{1}(x) = \varphi_{2}(x)$ for all $x \in J$. Let $x\in J$ be an arbitrary point. We choose a function $g\in \mathscr{C}^{k}(\Omega)$ such that $g(x)=1$ and the support of $g$, denoted by $\supp g$, is a compact subset of $J$. Consequently, for any $y \in \supp g$, we have $\varphi_{1}(y) = \varphi_{2}(y)$, which implies $\varphi_{1}(y)g(y) = \varphi_{2}(y)g(y)$, so $\varphi_{1}g = \varphi_{2}g$. In fact, for any positive integer $j$, we have $\varphi_{1}g^{j} = \varphi_{2}g^{j}$.

 We now substitute $f_{1}= \varphi_{1}$ and $f_{i}= g$ for $i=2, \ldots, n+1$ into equation \eqref{id_n}, and then do the same with $f_{1}= \varphi_{2}$ and $f_{i}= g$ for $i=2, \ldots, n+1$. Evaluating both resulting equations at the point $x$, where $g(x)=1$, we get:
 \[
  \sum_{i=0}^{n}(-1)^{i}\sum_{\mathrm{card}(I)=i, 1\notin I}\left(\prod_{j\in I}g(x)\right)D\left(\varphi_{1}(x)\prod_{k\in \left\{1\right\}\cup \left\{2, \ldots, n+1 \right\}\setminus I}g(x)\right) + 
  \]
 \[
  \sum_{i=0}^{n}(-1)^{i}\sum_{\mathrm{card}(I)=i, 1\in I}\left(\varphi_{1}(x)\prod_{j\in I\setminus \{1\}}g(x)\right)D\left(\prod_{k\in \left\{2, \ldots, n+1 \right\}\setminus I}g(x)\right) = 0
  \]
 and similarly for $\varphi_{2}$. Subtracting the second equation from the first, and using $g(x)=1$ and $\varphi_{1}(x) = \varphi_{2}(x)$, we obtain:
 \[
  \sum_{i=0}^{n}(-1)^{i}\sum_{\mathrm{card}(I)=i, 1\notin I} (D(\varphi_{1}g^{n+1-i})(x) - D(\varphi_{2}g^{n+1-i})(x)) + 
  \]
 \[
  \sum_{i=0}^{n}(-1)^{i}\sum_{\mathrm{card}(I)=i, 1\in I} \varphi_{1}(x) (D(g^{n+1-i})(x) - D(g^{n+1-i})(x)) = 0.
  \]
 The second sum is clearly zero. For the first sum, if $n+1-i \ge 1$, then $\varphi_{1}g^{n+1-i} = \varphi_{2}g^{n+1-i}$, so their $D$ values are equal. The only case left is when $n+1-i = 0$, i.e., $i = n+1$, which is not in the range of the sum.

 Let us consider the equation at $x$ more carefully.
 \[
  D(\varphi_{1}g^{n})(x) - \varphi_{1}(x)D(g^{n})(x) - \sum_{j=2}^{n+1} g(x)D(\varphi_{1}g^{n-1})(x) + \ldots + (-1)^{n} \prod_{j=2}^{n+1} g(x) D(\varphi_{1})(x) = 0
  \]
 and the same for $\varphi_{2}$. Subtracting and using the fact that $\varphi_{1}g^{j} = \varphi_{2}g^{j}$ at $x$ for $j \ge 1$, we are left with
 \[
  (-1)^{n} g(x)^{n} (D(\varphi_{1})(x) - D(\varphi_{2})(x)) = 0.
  \]
 Since $g(x)=1$, we have $(-1)^{n} (D(\varphi_{1})(x) - D(\varphi_{2})(x)) = 0$, which implies $D(\varphi_{1})(x) = D(\varphi_{2})(x)$. As $x\in J$ was arbitrary, we conclude that $D(\varphi_{1})\vert_{J}= D(\varphi_{2})\vert_{J}$. Thus $D$ is localized on intervals.
\end{proof}

In the next result we will recall a general property of operators localized on intervals, which is due to König and Milman.

\begin{prop}[König--Milman, Proposition 3.3]\label{prop_loc_pointwise}
 Let $k$ be a fixed nonnegative integer, $\Omega\subset \mathbb{R}$ be an open set and suppose that the operator $T\colon \mathscr{C}^{k}(\Omega)\to \mathscr{C}(\Omega)$ is localized on intervals, then there exists a function $F\colon D\times \mathbb{R}^{k+1}\to \mathbb{R}$ such that 
 \[
  (Tf)(x)= F(x, f(x), f'(x), \ldots, f^{(k)}(x))
 \]
for all $x\in D$ and for all $f\in \mathscr{C}^{k}(\Omega)$. 
\end{prop}

Our next lemma is a well-known formula for the $k$\textsuperscript{th}-derivative of the function $\ln \circ f$ for $f\in \mathscr{C}^{k}(\Omega)$ such that $f>0$.

\begin{lem}\label{lem_ln}
 Let $k$ be a positive integer, $\Omega\subset \mathbb{R}$ and $f\in \mathscr{C}^{k}(\Omega)$ be a positive function. Then 
 \begin{multline*}
  \diff{k}{x}\ln \circ f(x)
  \\
  = 
  \sum_{\substack{\sum_{i=1}^{k}i m_{i}=k}}\frac{k!}{m_{1}!\cdots m_{k}!} \cdot 
  \frac{(-1)^{m_{1}+\cdots+m_{k}-1}(m_{1}+\cdots+m_{k}-1)!}{f(x)^{m_{1}+\cdots+ m_{k}}} \cdot \prod_{1\leq j\leq  k} \left(\frac{f^{(j)}(x)}{j!}\right)^{m_{j}} 
  \\ 
  \left(x\in \Omega\right). 
 \end{multline*}
\end{lem}

\begin{dfn}
 Let $G$ and $S$ be commutative semigroups, and let $n$ be a positive integer. A function $A\colon G^{n}\to S$ is said to be \emph{$n$-additive} if, for each of its $n$ variables, it acts as a homomorphism from $G$ to $S$ when the other $n-1$ variables are held constant. When $n=1$, the function $A$ is simply called \emph{additive}, and when $n=2$, it is termed \emph{bi-additive}.
\end{dfn}

The \emph{diagonalization}, also known as the \emph{trace}, of an $n$-additive function $A\colon G^{n}\to S$ is a function $A^{\ast}\colon G\to S$ defined by evaluating $A$ at the same element $x \in G$ in all $n$ positions. That is, for any $x\in G$,
 \[
  A^{\ast}(x)=A\left(x, \ldots, x\right).
 \]
 
 \begin{dfn}
 Let $G$ and $S$ be commutative semigroups. A function $p\colon G\to S$ is called a \emph{generalized polynomial} if it can be expressed as a sum of diagonalizations of symmetric multi-additive functions from $G$ to $S$. More formally, a function $p\colon G\to S$ is a generalized polynomial if and only if there exists a nonnegative integer $n$ and, for each $k \in \{0, 1, \ldots, n\}$, a symmetric, $k$-additive function $A_{k}\colon G^{k}\to S$ such that $p(x) = \sum_{k=0}^{n}A^{\ast}_{k}(x)$ for all $x \in G$. In this case, we also say that $p$ is a generalized polynomial \emph{of degree at most $n$}.

 For a nonnegative integer $n$, functions $p_{n}\colon G\to S$ of the form
\[
 p_{n}= A_{n}^{\ast},
\]
where $A_{n}\colon G^{n}\to S$ is a symmetric and  $n$-additive function, are called \emph{generalized monomials of degree $n$}.
\end{dfn}

The next lemma is fundamental for us since it provides a solution of an equation of Aichinger in a general settings, which will appear later on during our study of equation \eqref{id_n}. We will use a result that was proven by Almira (see \cite{Alm23, Alm23b}).

\begin{lem}\label{lem_aichinger}
Let $S$ be a commutative cancellative semigroup and $H$ be a commutative group. Let $G=S-S$
be a natural extension of $S$. Assume that multiplication by $m!$ is bijective on $H$. Let $f\colon S\to H$ be a solution of 
\[
 f(x_{1}+\cdots+x_{m+1})= \sum_{i=1}^{m+1} g_{i}(x_{1}, \ldots, x_{i-1}, \hat{x_{i}}, x_{i-1}, \ldots, x_{m+1}) 
 \qquad 
 \left(x_{1}, \ldots, x_{m+1}\in S\right)
\]
with certain functions $g_{i}\colon S^{m}\to H$, where for all $i=1, \ldots, m+1$, the symbol $\hat{x_{i}}$ denotes that the function $g_{i}$ does not depend on the variable $x_{i}$. Then there exists a generalized polynomial $F\colon G\to H$ of degree at most $m$, such that $F\vert_{S}=f$ and the functions $g_{i}$ are also generalized polynomials of degree at most $m$. 
\end{lem}

\begin{cor}\label{cor_aichinger}
 Let $k$ and $m$ be nonnegative integers and $f\colon \mathbb{R}^{k+1}\to \mathbb{R}$  be a function such that 
 \[
 f(x_{1}+\cdots+x_{m+1})= \sum_{i=1}^{m+1} g_{i}(x_{1}, \ldots, x_{i-1}, \hat{x_{i}}, x_{i-1}, \ldots, x_{m+1}) 
 \qquad 
 \left(x_{1}, \ldots, x_{m+1}\in \mathbb{R}^{k+1}\right)
\]
holds with certain functions $g_{i}\colon \mathbb{R}^{m(k+1)}\to \mathbb{R}$. If the function $f$ is continuous at a point $x^{\ast}\in \mathbb{R}^{k+1}$, then $f$ is a $(k+1)$-variable ordinary polynomial of degree at most $m$. 
\end{cor}

Now, we are ready to state and prove our main result.

\begin{thm}\label{thm_main}
 Let $k$ be a fixed nonnegative integer and $\Omega\subset \mathbb{R}$ be a nonempty open set.
 Suppose that the operator  $D\colon \mathscr{C}^{k}(\Omega)\to \mathscr{C}(\Omega)$ satisfies \eqref{id_n} for all $f_{1}, \ldots, f_{n+1}\in \mathscr{C}^{k}(\Omega)$.
 Then there exist functions $c_{1}, \dots , c_{n}, d_1, \dots , d_{n}\in \mathscr{C}^{k}(\Omega)$ such that 
\begin{align}
 D(f)(x) &= c_1(x)f'(x)+ c_2(x) f''(x) + \dots + c_n(x) f^{(n)}(x) \nonumber \\ &+ d_1(x) f(x) (\ln (f(x)) )+ d_2(x) f(x) (\ln f(x) )^2+\dots  + d_n(x) f(x) (\ln f(x) )^n,  
\label{D}
\end{align} 
holds for all $f\in \mathscr{C}^{k}(\Omega)$ and $x\in \Omega$. 
Further, if $k<n$, then $c_{k+1}=c_{k+2}=\dots = c_{n}=0$.

Conversely, the operator $D$ given by the formula \eqref{D} satisfies \eqref{id_n}. 
\end{thm}

\begin{proof}
Suppose that the operator $D\colon \mathscr{C}^{k}(\Omega)\to \mathscr{C}(\Omega)$ satisfies equation \eqref{id_n}. By Lemma \ref{lem_loc_interval}, the operator $D$ is localized on intervals. Applying Proposition \ref{prop_loc_pointwise}, this localization property implies that for any function $f\in \mathscr{C}^{k}(\Omega)$ and any point $x\in \Omega$, the value of $D(f)(x)$ depends only on the values of $f$ and its derivatives up to order $k$ at the point $x$. Therefore, there exists a function $F\colon \Omega\times \mathbb{R}^{k+1}\to \mathbb{R}$ such that 
\[
 D(f)(x)= F(x, f(x), \ldots, f^{(k)}(x))
\]
for all $x\in \Omega$ and for all $f\in \mathscr{C}^{k}(\Omega)$. 

To determine the specific form of $D$, we consider the operator $P\colon \mathscr{C}^{k}(\Omega)\to \mathscr{C}(\Omega)$ defined by 
\[
 P(g)(x)= \frac{D(\exp\circ g)(x)}{\exp\circ g(x)} 
 \qquad 
 \left(x\in \Omega, g\in \mathscr{C}^{k}(\Omega)\right). 
\]
Since $D$ is pointwise localized, the operator $P$ is also pointwise localized. Thus, there exists a function $G\colon \Omega \times \mathbb{R}^{k+1}\to \mathscr{C}(\Omega)$ such that 
\[
 P(g)(x)= G(x, g(x), \ldots, g^{(k)}(x)) 
 \qquad 
 \left(x\in \Omega, g\in \mathscr{C}^{k}(\Omega)\right). 
\]
Let $g_{1}, \ldots, g_{n+1}\in \mathscr{C}^{k}(\Omega)$ be arbitrary functions. Substituting $f_j = \exp(g_j)$ into equation \eqref{id_n}, we get 
\begin{align*}
 0 &= \sum_{i=0}^{n}(-1)^{i}\sum_{\mathrm{card}(I)=i}\left(\prod_{k\in I}\exp\circ g_{k}\right)D\left(\prod_{k\notin I}\exp \circ g_{k}\right) \\
 &= \left(\prod_{j=1}^{n+1}\exp\circ g_{j}\right) \cdot  \sum_{i=0}^{n}(-1)^{i}\sum_{\mathrm{card}(I)=i}\frac{D\left(\prod_{k\notin I}\exp \circ g_{k}\right)}{\prod_{k\notin I}\exp\circ g_{k}} \\
 &= \left(\prod_{j=1}^{n+1}\exp\circ g_{j}\right) \cdot  \sum_{i=0}^{n}(-1)^{i}\sum_{\mathrm{card}(I)=i}\frac{D\left(\exp\left(\sum_{k\notin I}g_{k}\right)\right)}{\exp\left(\sum_{k\notin I}g_{k}\right)} \\
 &= \left(\prod_{j=1}^{n+1}\exp\circ g_{j}\right) \cdot  \sum_{i=0}^{n}(-1)^{i}\sum_{\mathrm{card}(I)=i}P\left(\sum_{k\notin I}g_{k}\right). 
\end{align*}
Since $\exp(x) > 0$ for all $x$, we have $$\sum_{i=0}^{n}(-1)^{i}\sum_{\mathrm{card}(I)=i}P\left(\sum_{k\notin I}g_{k}\right) = 0.$$ 
For any $v_{j}= (v_{j}^{l})_{l=0}^{k}\in \mathbb{R}^{k+1}$, $j=1, \ldots, n+1$ and $x\in \Omega$ there exists functions $g_{1}, \ldots, g_{n+1}\in \mathscr{C}^{k}(\Omega)$ with 
\[
 g_{j}^{(l)}(x)= v_{j}^{l}. 
\]
Therefore, the function $G$ satisfies 
\[
 \sum_{i=0}^{n}(-1)^{i}\sum_{\mathrm{card}(I)=i} G(x, \sum_{j\notin I}v_{j})=0
\]
for all $x\in \Omega$ and for all $v_{1}, \ldots, v_{n+1}\in \mathbb{R}^{k+1}$.
This and Lemma \ref{lem_aichinger} imply that for each fixed $x\in \Omega$, the mapping $\mathbb{R}^{k+1}\ni v \mapsto G(x, v)$ is a generalized polynomial of degree at most $n$. 
The continuity of $P(g)$ ensures the continuity of $x \mapsto G(x, g(x), \ldots, g^{(k)}(x))$, which means the coefficients of this polynomial depend continuously on $x$. Thus, we can write
\[
 G(x, v)= \sum_{|\alpha|\leq n} c_{\alpha}(x)v^{\alpha} 
 \qquad 
 \left(x\in \Omega, v\in \mathbb{R}^{k+1}\right)
\]
with some functions $c_{\alpha}\in \mathscr{C}(\Omega)$ and for all multi-index $\alpha\in \mathbb{R}^{k+1}$ with $|\alpha|\leq n$.

Now, let $f\in \mathscr{C}^{k}(\Omega)$ be a positive function. Then $f= \exp \circ g$ for some $g = \ln \circ f \in \mathscr{C}^{k}(\Omega)$. We have 
\begin{align*}
 D(f)(x) &= f(x)\cdot P(\ln\circ f)(x) \\
 &= f(x) \cdot \sum_{i=0}^{n}\sum_{\substack{|\alpha|=i\\ \alpha \in \mathbb{N}_{0}^{k+1}}} 
 c_{\alpha}(x)\left((\ln f)(x)\right)^{\alpha_{0}}\cdot \left((\ln f)'(x)\right)^{\alpha_{1}} \cdots \left((\ln f)^{(k)}(x)\right)^{\alpha_{k}}.
\end{align*}

If $f\in \mathscr{C}^{k}(\Omega)$ and $f(x)<0$ at some $x\in \Omega$, then there exists an open interval $J\subset \Omega$ containing the point $x$ such that $f\vert_{J}< 0$. Thus if $g\in \mathscr{C}^{k}(\Omega)$ and $g(x)<0$ for all $x\in \Omega$, then $f\vert_{J}= g\vert_{J}$. Consequently, Lemma \ref{lem_loc_interval} yields that $D(f)(x)= D(g)(x)$. Therefore, without the loss of generality
we can assume that $f<0$ on $\Omega$. Then $f=-|f|$ and if we substitute in \eqref{id_n}
\[
 f_{1}= |f| 
 \qquad 
 f_{i}= -\mathbf{1} 
 \qquad 
 \text{ for } \qquad  i\geq 2,
\]
 then  Lemma \ref{lem_loc_interval} gives us $D(-|f|)=-D(|f|)$. Thus we have 
\begin{multline*}
 D(f)(x)
 = f(x) \cdot \sum_{i=0}^{n}\sum_{\substack{|\alpha|=i\\ \alpha \in \mathbb{N}_{0}^{k+1}}} 
 c_{\alpha_{0}, \ldots, \alpha_{k}}(x)\left((\ln\circ |f|)(x)\right)^{\alpha_{0}}\cdot \left((\ln \circ |f|)'(x)\right)^{\alpha_{1}} \cdots \left((\ln \circ |f|)^{(k)}(x)\right)^{\alpha_{k}}
\end{multline*}
for all $f\in \mathscr{C}^{k}(\Omega)$ with $f(x)\neq 0$ for $x\in \Omega$. 

For $D(f)$ to be defined and continuous on $\mathscr{C}^{k}(\Omega)$, including points where $f(x)=0$, we analyze the behavior of the terms involving logarithms and their derivatives. To ensure continuity of $D(f)$ at $f(x)=0$, singular terms must cancel. We will justify this by expanding each term of the form 
$$ \left((\ln \circ |f|)^{(j)}(x)\right)^{\alpha_{j}} $$
for $j = 0, 1 , \dots , n$ using Lemma \ref{lem_ln}. Since $D(f) \in \mathscr{C}(\Omega)$, then every singular term must vanish, i.e. must be cancelled by some other term with the same order of singularity. Therefore, after regrouping, we will obtain a representation of $D(f)$ that contains only continuous multipliers of the following summands:
\begin{itemize}
	\item $f^{(k)}$ for $k=1, \dots , n$, which corresponds to the case $m_0 = \dots = m_{k-1}=0$, $m_k=1$ in Lemma \ref{lem_ln} (singularity $O(1/f^{m_1+\dots + m_k})$ is reduced with the term $f(x)$ in front of the above representation of $D(f)$ if and only if  $m_0 = \dots = m_{k-1}=0$ and $m_k=1$),
	\item $\ln \circ ( |f|)^j$, which corresponds to the case  $\alpha_0=j$, $\alpha_1=\dots =\alpha_j=0$, and no singularity is produced.
\end{itemize}

The above discussion leads to the conclusion that $D(f)$ must have the form given by equation \eqref{D}, where $c_1, \dots , c_n, d_1, \dots , d_n \in \mathscr{C}(\Omega)$. The terms with logarithms are continuous at $f=0$ due to $\lim_{x\to 0}x\ln(|x|)^{k}=0$, with the convention $0\ln(0)^k=0$.

The converse, proving that the operator $D$ defined by \eqref{D} satisfies \eqref{id_n}, involves a direct calculation using the properties of derivatives and logarithms.
\end{proof}

Our main result yields two immediate corollaries. The first characterizes linear operators that satisfy the identity~\eqref{id_n}, while the second describes operators satisfying~\eqref{id_n} that also annihilate polynomials up to a certain degree.

\begin{cor}\label{cor_linear}
 Let $k$ be a fixed nonnegative integer and let $\Omega\subset \R$ be a nonempty open set.
 Suppose the operator $D\colon \mathscr{C}^{k}(\Omega)\to \mathscr{C}(\Omega)$ is \emph{linear} and satisfies identity~\eqref{id_n}
 for all $f_{1}, \ldots, f_{n+1}\in \mathscr{C}^{k}(\Omega)$.
 Then there exist functions $c_{1}, \dots , c_n\in \mathscr{C}^{k}(\Omega)$ such that for every function $f\in \mathscr{C}^{k}(\Omega)$, operator $D$ act as a linear differential operator:
 \begin{equation}
 D(f)(x)= c_{1}(x)f'(x)+c_{2}(x)f''(x)+ \dots + c_n(x)f^{(n)}(x).
 \label{D2}
 \end{equation}
 Moreover, if $k<n$, the coefficients corresponding to derivatives of order higher than $k$ must vanish, i.e., $c_{k+1}=c_{k+2}=\dots = c_{n}=0$.

 Conversely, any operator $D$ defined by~\eqref{D2} for $f\in \mathscr{C}^{k}(\Omega)$ is linear and satisfies~\eqref{id_n}.
\end{cor}

\begin{cor}\label{cor_annihilate} 
 Let $k, n$ and $j$ be a fixed nonnegative integers and let $\Omega\subset \R$ be a nonempty open set.
 Suppose the operator $D\colon \mathscr{C}^{k}(\Omega)\to \mathscr{C}(\Omega)$ satisfies identity~\eqref{id_n} 
 and, additionally, annihilates all polynomials of degree at most $j$. 
 Then there exist functions $c_{j+1}, \dots, c_n\in \mathscr{C}(\Omega)$ such that for all $f\in \mathscr{C}^{k}(\Omega)$:
 \begin{equation}
 D(f)(x)= c_{j+1}(x)f^{(j+1)}(x)+ c_{j+2}(x)f^{(j+2)}(x)+ \dots + c_{n}(x)f^{(n)}(x).
 \label{D3}
 \end{equation}
 Furthermore, if $k\leq j$, then $D$ is the zero operator on $\mathscr{C}^{k}(\Omega)$ and if $k<n$, then $c_{k+1}= \cdots = c_{n}=0$.

 Conversely, any operator $D\colon \mathscr{C}^{k}(\Omega)\to \mathscr{C}(\Omega)$  given by formula~\eqref{D3} satisfies identity~\eqref{id_n} and annihilates all polynomials of degree at most $j$.
\end{cor}

We now aim to extend Corollary~\ref{cor_linear}. Instead of identity~\eqref{id_n}, which involves $n+1$ function arguments, we will consider identity~\eqref{id_single}. This latter identity involves only a single function $f \in \mathscr{C}^{k}(\Omega)$, potentially offering greater applicability. To proceed, we first introduce the notion of difference operators.

Let $T\colon \mathscr{C}^{k}(\Omega)\to \mathscr{C}(\Omega)$ be an operator and let $h\in \mathscr{C}^{k}(\Omega)$ be a function. The action of the (first-order) difference operator $\Delta_{h}$ on $T$ is defined as
\[
 \Delta_{h}T(f)= T(f+h)-T(f)
 \qquad
 \left(f\in \mathscr{C}^{k}(\Omega)\right).
\]
For a positive integer $m$ and functions $h_{1}, \ldots, h_{m}\in \mathscr{C}^{k}(\Omega)$, the $m$\textsuperscript{th} order difference operator $\Delta_{h_{1}, \ldots, h_{m}}$ is defined recursively by 
\[
 \Delta_{h_{1}, h_{2}, \ldots, h_{m}}T= \Delta_{h_{1}}\left(\Delta_{h_{2}}\left(\cdots \Delta_{h_{m}}T\right)\right).
\]

With these definitions, we can now present a version of Corollary~\ref{cor_linear} that relies on the single-function condition~\eqref{id_single}.

\begin{prop}\label{prop_generalization} 
 Let $k$ be a fixed nonnegative integer and let $\Omega\subset \R$ be a nonempty open set.
 Suppose the operator $D\colon \mathscr{C}^{k}(\Omega)\to \mathscr{C}(\Omega)$ is \emph{linear} and satisfies identity~\eqref{id_single}.
 Then there exist functions $c_{1}, \dots , c_n\in \mathscr{C}^{k}(\Omega)$ such that $D$ has the form~\eqref{D2} for all $f\in \mathscr{C}^{k}(\Omega)$.
 Furthermore, if $k<n$, then $c_{k+1}=c_{k+2}=\dots = c_{n}=0$.
\end{prop}

\begin{proof}
 The strategy is to demonstrate that if a linear operator $D\colon \mathscr{C}^{k}(\Omega)\to \mathscr{C}(\Omega)$ satisfies equation~\eqref{id_single}, it must also satisfy equation~\eqref{id_n}. The conclusion will then follow directly from Corollary~\ref{cor_linear}.

 Consider the multi-additive, symmetric operator $\mathcal{A}_{n+1}\colon \prod_{i=1}^{n+1}\mathscr{C}^{k}(\Omega)\to \mathscr{C}(\Omega)$ defined by
 \begin{multline*}
  \mathcal{A}_{n+1}(f_1, f_2, \dots , f_{n+1})= \\
   \sum_{i=0}^{n}(-1)^{i}\sum_{\substack{I \subset \{1, \dots, n+1\} \\ \mathrm{card}(I)=i}}\left(\prod_{j\in I}f_{j}\right)D\left(\prod_{k\in \left\{1, \ldots, n+1 \right\}\setminus I}f_{k}\right)
  \qquad
  \left(f_1, \dots , f_{n+1}\in \mathscr{C}^{k}(\Omega)\right).
 \end{multline*}
 Let $\mathcal{A}^\ast_{n+1}(f) = \mathcal{A}_{n+1}(f, f,\dots , f)$ denote the diagonalization of $\mathcal{A}_{n+1}$. The condition that $D$ satisfies equation~\eqref{id_single} is equivalent to stating that $\mathcal{A}^\ast_{n+1}(f) = 0$ for all $f \in \mathscr{C}^{k}(\Omega)$:
 \[
 \mathcal{A}^\ast_{n+1}(f) =
  \sum_{i=0}^{n}(-1)^{i}\binom{n+1}{i}f^{i} D\left(f^{n+1-i}\right)=0
 \qquad \forall f\in \mathscr{C}^{k}(\Omega).
 \]
 The polarization formula (see, e.g., \cite[Lemma 1.4]{Sze91}) asserts that for a symmetric and multi-additive operator, the diagonalization is identically zero if and only if the operator itself is identically zero. Thus, $\mathcal{A}^\ast_{n+1} \equiv 0$ implies $\mathcal{A}_{n+1} \equiv 0$.
 Specifically, $\mathcal{A}_{n+1}(f_1, \dots, f_{n+1}) = 0$ for all $f_1, \dots, f_{n+1} \in \mathscr{C}^{k}(\Omega)$. This is precisely the statement that the linear operator $D$ satisfies identity~\eqref{id_n}.
 Applying Corollary~\ref{cor_linear} yields the desired form~\eqref{D2} for $D$, completing the proof.
\end{proof}

\bibliographystyle{plain}
\bibliography{GseFec24b_bib}

\end{document}